\documentclass[12pt]{amsart}
\usepackage{amsmath,amssymb}
\usepackage{amsfonts}
\usepackage{amsthm}
\usepackage{latexsym}
\usepackage{graphicx}

\def\p{\partial}

\def\R{\mathbb{R}}

\def\vv<#1>{\langle#1\rangle}

\def\XXint#1#2{\setbox0=\hbox{$#1{#2}{\int}$}{#2}\kern-.5\wd0 }

\def\XXint#1#2#3{{\setbox0=\hbox{$#1{#2#3}{\int}$}
     \vcenter{\hbox{$#2#3$}}\kern-.5\wd0}}

\def\vv<#1>{{\left\langle#1\right\rangle}}

\def\Vol{\mbox{Vol}}
\def\sph{\mathbb{S}}
\newtheorem{thm}{Theorem}[section]

\newtheorem{lem}{Lemma}[section]

\theoremstyle{definition}

\theoremstyle{remark}

\numberwithin{equation}{section}

\begin{document}
\title{Estimates for higher Steklov eigenvalues}

\author{Liangwei Yang}
\address{Department of Mathematics, Shantou University, Shantou, Guangdong, 515063, China}
\email{13lwyang@stu.edu.cn}
\author{Chengjie Yu$^1$}
\address{Department of Mathematics, Shantou University, Shantou, Guangdong, 515063, China}
\email{cjyu@stu.edu.cn}
\thanks{$^1$Research partially supported by a supporting project from the Department of Education of Guangdong Province with contract no. Yq2013073, the Yangfan project from Guangdong Province and NSFC 11571215.}
\renewcommand{\subjclassname}{%
  \textup{2010} Mathematics Subject Classification}
\date{}
\begin{abstract}
In this paper, motivated by the work of Raulot and Savo, we generalize Raulot-Savo's estimate for the first Steklov eigenvalues of Euclidean domains to higher Steklov eigenvalues.
\end{abstract}
\maketitle\markboth{Yang \& Yu}{Higher Steklov Eigenvalues}
\section{Introduction}
Let $u$ be a smooth function on the boundary $\p M$ of a compact orientable Riemannian manifold $(M^n,g )$ with nonempty boundary and $\hat u$ be its harmonic extension on $M$. Then, the Dirichlet-to-Neumann map or Steklov operator maps $u$ to $\frac{\p \hat u}{\p \nu}$ where $\nu$ is the outward unit normal vector on $\p M$. The spectrum of the Dirichlet-to-Neumann map is discrete (See \cite{Ta}). They are called Steklov eigenvalues of $(M,g)$. The Steklov eigenvalues have been extensively studied. For example \cite{Br,GP1,GP2,GP3,HPS,WX,W,FS1,FS2,CEG,E} obtained interesting estimates for them. \cite{GP} makes an excellent survey for recent progresses of Steklov eigenvalues.

Recently, Raulot and Savo \cite{RS2} extended the notion of Steklov eigenvalues to differential forms. Let $\omega$ be a differential $p$-form on $\p M$. Let $\hat \omega$ be the tangential harmonic extension of $\omega$. That is to say, $\hat \omega$ satisfies:
\begin{equation}
\left\{\begin{array}{l}\Delta\hat\omega=0\\
\iota^*\hat\omega=\omega\\
i_\nu\omega=0.
\end{array}\right.
\end{equation}
Here $\iota:\p M\to M$ is the natural inclusion. Then, the Dirichlet-to-Neumann map for $p$-forms defined by Raulot-Savo \cite{RS2} maps $\omega$ to $i_\nu d\hat\omega$. It is clear that, when $p=0$, the Dirichlet-to-Neumann map defined by Raulot and Savo coincides with the classical Dirichlet-to-Neumann map for functions. It was shown in \cite{RS2} that the spectrum of the Dirichlet-to-Neumann map is also discrete. We can arrange them in ascending order(counting multiplicity) as:
$$0\leq\sigma_1^{(p)}\leq\sigma_2^{(p)}\leq\cdots\leq\sigma_k^{(p)}\leq\cdots.$$
They are called Steklov eigenvalues for $p$-forms.

In \cite{RS1,RS2,RS3}, Raulot and Savo obtained the following interesting upper bounds for Steklov eigenvalues.
\begin{thm}[Raulot-Savo \cite{RS1,RS2,RS3}]\label{thm-RS}
\begin{enumerate}
\item Let $(M^n,g)$ be a compact orientable manifold with nonempty boundary. Then,
    $$\sigma_1^{(n-1)}(M)\leq \frac{\Vol(\p M)}{\Vol(M)};$$
\item Let $\Omega$ be a bounded domain with smooth boundary in $\R^n$. Then
\begin{enumerate}
\item $$\sigma_{2}^{(0)}(\Omega)\leq \frac{1}{n}\frac{\Vol(\p \Omega)}{\Vol(\Omega)}.$$
    The equality holds if and only if $\Omega$ is a ball;
\item When $1\leq p<\frac{n}{2}$,
$$\sigma_1^{(p)}(\Omega)< \frac{p+1}{n}\frac{\Vol(\p\Omega)}{\Vol(\Omega)};$$
\item When $p\geq \frac{n}{2}$,
$$\sigma_1^{(p)}(\Omega)\leq \frac{p+1}{n}\frac{\Vol(\p\Omega)}{\Vol(\Omega)}.$$
The equality holds if and only if $\Omega$ is ball.
\end{enumerate}
\end{enumerate}
\end{thm}

Other interesting estimates of Steklov eigenvalues for differential forms are also obtained in \cite{K}. The estimate (a) is also obtain in \cite{IM} by a different method.

In this paper,motivated by the work of Raulot and Savo, we obtain the follow general estimate for Steklov eigenvalues.
\begin{thm}\label{thm-para}
Let $(M^n,g)$ be a compact orientable Riemannian manifold with nonempty boundary. Let $V$ be the space of parallel exact $1$-forms on $M$. Suppose that $m=\dim V>0$. Then
\begin{equation}\label{eqn-para}
\sigma_{k+b_{p-1}}^{(p-1)}(M)\leq \frac{C_{m-1}^{p-1}}{C_m^p+1-k}\frac{\Vol(\p M)}{\Vol(M)}
\end{equation}
for $p=1,2,\cdots,m$ and $k=1,2,\cdots,C_m^p$. Here $C_m^p=\frac{m(m-1)\cdots(m-p+1)}{p!}$ and $b_k$ denotes the $k$-th Betti number of $M$.

\end{thm}

When $M$ is a domain in $\R^n$, $\dim V=n$, the theorem above give some generalized estimates with respect to Raulot-Savo's result.
\begin{thm}\label{thm-Omega}
Let $\Omega$ be a bounded domain with smooth boundary in $\R^n$. Then
\begin{equation}\label{eqn-Omega}
\sigma_{k+b_{p}}^{(p)}(\Omega)\leq \frac{C_{n-1}^{p}}{C_n^{p+1}+1-k}\frac{\Vol(\p \Omega)}{\Vol(\Omega)}
\end{equation}
for $p=0,1,\cdots,n-1$ and $k=1,2,\cdots,C_{n}^{p+1}$. Here $b_p$ denotes the $p$-th Betti number of $\Omega$. Moreover, when $k=1$ and $p\geq \frac{n}{2}$, the equality of \eqref{eqn-Omega} holds if and only if $\Omega$ is a ball.
\end{thm}

In \cite{RS1}, Raulot and Savo also obtained all the Steklov values of the unit ball in $\R^n$. The first several Steklov eigenvalues are as follows:
\begin{enumerate}
\item $\sigma_1^{(0)}=0,\sigma_2^{(0)}=\cdots=\sigma_{n+1}^{(0)}=1,\cdots$;
\item When $1\leq p<\frac{n}{2}$,
\begin{equation}\label{eqn-eigen-ball-1}
\sigma_1^{(p)}=\cdots=\sigma_{C_n^{p}}^{(p)}=\frac{(n+2)p}{n},\sigma_{C_n^{p}+1}^{(p)}=\cdots=\sigma_{C_n^p+C_n^{p+1}}^{(p)}=p+1,\cdots;
\end{equation}
\item When $\frac{n}{2}\leq p\leq n-1$,
\begin{equation}\label{eqn-eigen-ball-2}
\sigma_1^{(p)}=\cdots=\sigma_{C_n^{p+1}}^{(p)}=p+1,\sigma_{C_n^{p+1}+1}^{(p)}=\cdots=\sigma_{C_n^p+C_n^{p+1}}^{(p)}=\frac{(n+2)p}{n},\cdots;
\end{equation}
\end{enumerate}

From this, we know that the estimate for higher Steklov eigenvalues in \eqref{eqn-Omega} is not sharp for unit balls. However, when $\Omega$ is a strictly convex domain in $\R^n$, we have the following sharp estimate for higher Steklov eigenvalues.
\begin{thm}\label{thm-convex}
Let $\Omega$ be a strictly convex bounded domain in $\R^n$ with smooth boundary. Then,
\begin{enumerate}
\item \begin{equation}
\sigma_{k}^{(0)}(\Omega)\leq \frac{1}{n\min_{\p\Omega} K}\cdot \frac{\Vol(\sph^{n-1})}{ \Vol(\Omega)}
\end{equation}
for $k=2,3,\cdots,n+1$. The equality holds for some $k=2,3,\cdots,n+1$ if and only if $\Omega$ is a ball; \item When $1\leq p<\frac{n}{2}$,
\begin{equation}
\sigma_k^{(p)}(\Omega)< \frac{p+1}{n\min_{\p\Omega}K}\frac{\Vol(\sph^{n-1})}{\Vol(\Omega)}
\end{equation}
for $k=1,2,\cdots,C_n^p$ and
\begin{equation}
\sigma_k^{(p)}(\Omega)\leq \frac{p+1}{n\min_{\p\Omega}K}\frac{\Vol(\sph^{n-1})}{\Vol(\Omega)}
\end{equation}
for $k=C_n^p+1,\cdots,C_n^{p+1}$. The equality holds for some $k=C_n^p+1,\cdots,C_n^{p+1}$ if and only if $\Omega$ is a ball;
\item When $\frac{n}{2}\leq p\leq n-1$,
\begin{equation}
\sigma_k^{(p)}(\Omega)\leq \frac{p+1}{n\min_{\p\Omega}K}\frac{\Vol(\sph^{n-1})}{\Vol(\Omega)}
\end{equation}
for $k=1,2,\cdots,C_n^{p+1}$. The equality holds for some $k=1,2,\cdots,C_n^{p+1}$ if and only if $\Omega$ is a ball.
\end{enumerate}
Here $K$ is the Gaussian curvature of $\p\Omega$.
\end{thm}

The remaining part of the paper is organized as follows. In section 2, we prove the general estimate \eqref{eqn-para}. In section 3, we prove Theorem \ref{thm-convex}.
\section{A general estimate}
In this section, we prove Theorem \ref{thm-para}. The following lemma is the key step to the proof our estimate which is motivated by Raulot-Savo \cite{RS1}.
\begin{lem}\label{lem-comp}
Let $(M^n,g)$ be a compact orientable Riemannian manifold with nonempty boundary and $V$ be a linear subspace of the space of exact parallel $p$-forms on $M$ for $p\geq 1$. Suppose that $\dim V=m>0$. Then
\begin{equation}
\sigma_{k+b_{p-1}}^{(p-1)}(M)\leq v_k
\end{equation}
for $k=1,2,\cdots,m$. Here $v_1\leq v_2\leq\cdots\leq v_m$ are the eigenvalues of the linear transformation $A$ on $V$ with
\begin{equation}
\int_M\vv<A\xi,\eta>dV_M=\int_{\p M}\vv<i_\nu\xi,i_\nu\eta>dV_{\p M}
\end{equation}
for any $\xi$ and $\eta$ in $V$.
\end{lem}
\begin{proof}
Let $\xi_1,\xi_2,\cdots,\xi_m$ be the eigenforms of $v_1,v_2,\cdots,v_m$ respectively. It is clear that we can assume that
\begin{equation}
\int_M\vv<\xi_i,\xi_j>dV_M=0
\end{equation}
for $i\neq j$. Then
\begin{equation}
\int_{\p M}\vv<i_\nu\xi_i,i_\nu\xi_j>dV_{\p M}=\int_M\vv<A\xi_i,\xi_j>dV_M=v_i\int_M\vv<\xi_i,\xi_j>dV_M=0
\end{equation}
when $i\neq j$. Let $\theta_i\in A^{p-1}(M)$ be such that
$$d\theta_i=\xi_i.$$
By the Hodge decomposition for compact Riemannian manifolds with nonempty boundary (See \cite{S}), we have
\begin{equation}
\theta_i=d\alpha_i+\delta\beta_i+\gamma_i
\end{equation}
where $\beta_i\in A^{p}(M)$ with $i_\nu\beta_i=0$ and $\gamma_i\in \mathcal H_N^{p-1}(M)$. Here
\begin{equation}
\mathcal H_N^{r}(M)=\{\gamma\in A^r(M)\ |\ d\gamma=\delta\gamma=0\ \mbox{and }i_\nu\gamma=0\}.
\end{equation}
Choose $\tilde \gamma_i\in\mathcal H_N^{p-1}(M)$, such that
\begin{equation}
\omega_i=\delta\beta_i+\tilde \gamma_i \perp_{L^2(\p M)}\mathcal H_N^{p-1}(M).
\end{equation}
Then $d\omega_i=\xi_i$, $i_\nu\omega_i=0$ and $\delta\omega_i=0$. Moreover, since $\xi_i$ is parallel, $\delta\xi_i=0$. So,
\begin{equation}
\Delta\omega_i=d\delta\omega_i+\delta d\omega_i=0.
\end{equation}

Let
$$E_k= \mbox{span}\{\omega_1,\omega_2,\cdots,\omega_k\}$$
for $k=1,2,\cdots,m$. Then $\dim E_k=k$ and $E_k\perp_{L^2(\p M)}\mathcal H_N^{p-1}(M)$. Let  $\alpha_1,\alpha_2,\cdots,\alpha_k,\dots$ be a complete orthonormal system of eigenforms for positive eigenvalues for the Dirichlet-to-Neumann map on $A^{p-1}(\p M)$. By dimension reasons, we know that
\begin{equation}
E_k\cap \overline{\mbox{span}\{\hat\alpha_k,\hat\alpha_{k+1},\cdots\}}\neq 0.
\end{equation}
Let $\omega\in E_k\cap \overline{\mbox{span}\{\hat\alpha_k,\hat\alpha_{k+1},\cdots\}}$ with $\omega\neq 0$.

Suppose that $\omega=\sum_{i=k}^\infty c_i\hat \alpha_i$, then
\begin{equation}\label{eqn-1}
\begin{split}
\frac{\int_{\p M}\|i_\nu d\omega\|^2 dV_{\p M}}{\int_M \|d\omega\|^2dV_M}\geq &\frac{\int_{\p M}\|i_\nu d\omega\|^2 dV_{\p M}}{\int_M (\|d\omega\|^2+\|\delta\omega\|^2)dV_M}\\
=&\frac{\sum_{i=k}^\infty {\sigma_{i+b_{p-1}}^{(p-1)}}^2c_i^2}{\sum_{i=k}^\infty {\sigma_{i+b_{p-1}}^{(p-1)}}c_i^2}\\
\geq& \sigma_{k+b_{p-1}}^{(p-1)}.
\end{split}
\end{equation}

On the other hand, suppose that $\omega=\sum_{i=1}^ka_i\omega_i$. Then
\begin{equation}\label{eqn-2}
\begin{split}
\frac{\int_{\p M}\|i_\nu d\omega\|^2 dV_{\p M}}{\int_M\|d\omega\|^2dV_M}=\frac{\sum_{i=1}^kv_ia_i^2\int_M\|\xi_i\|^2dV_M}{\sum_{i=1}^ka_i^2\int_M\|\xi_i\|^2dV_M}\leq v_k.
\end{split}
\end{equation}
Combining \eqref{eqn-1} and \eqref{eqn-2}, we obtain the conclusion.
\end{proof}
The following lemma will be needed to obtain estimates of $v_k$.
\begin{lem}\label{lem-int}
Let $H$ be a linear subspace of $\R^n$ of dimension $m$. Let $f(x)=a_1x_1+\cdots+a_nx_n$ be a linear function on $\R^n$. Then
\begin{equation}
\int_{\sph^{n-1}\cap H}f^2(x)dV_{\sph^{n-1}\cap H}(x)\leq\frac{\Vol(\sph^{m-1})\sum_{i=1}^na_i^2}{m}.
\end{equation}
The equality holds if and only if $(a_1,a_2,\cdots,a_n)\in H$.
\end{lem}
\begin{proof}
Without loss of generality, we can assume that
$$H=\{(x_1,x_2,\cdots,x_n)\ |\ x_{m+1}=x_{m+2}=\cdots=x_n=0\}.$$
Then,
\begin{equation}
\begin{split}
\int_{\sph^{n-1}\cap H}f^2(x)dV_{\sph^{n-1}\cap H}(x)=&\sum_{k=1}^m a_k^2\int_{\sph^{m-1}}x_k^2dx\\
=&\frac{\sum_{k=1}^m a_k^2\Vol(\sph^{m-1})}{m}\\
\leq&\frac{\Vol(\sph^{m-1})\sum_{k=1}^n a_k^2}{m}.\\
\end{split}
\end{equation}
Equality holds only when $a_{m+1}=a_{m+2}=\cdots=a_n=0$. This completes the proof.
\end{proof}
By Lemma \ref{lem-comp}, if we can estimate of the eigenvalues $v_k$, then we get estimate for Steklov eigenvalues. So, we come to estimate $v_k$.
\begin{lem}\label{lem-v}
Let $(M^n,g)$ be a compact orientable Riemannian manifold with nonempty boundary. Let $V$ be the space of parallel exact 1-forms on $M$. Suppose that $\dim V=m>0$. Then,
\begin{equation}
v_k^{(p)}\leq \frac{C_{m-1}^{p-1}}{C_m^p+1-k}\frac{\Vol(\p M)}{\Vol(M)}
\end{equation}
for $k=1,2,\cdots,C_m^p$ and $p=1,2,\cdots,n$. Here $v_1^{(p)}\leq v_2^{(p)}\leq \cdots\leq  v_{C_m^p}^{(p)}$ is the eigenvalues of the linear transformation $A^{(p)}$ on $\wedge^p V$ with
\begin{equation}
\int_M\vv<A^{(p)}\xi,\eta>dV_M=\int_{\p M}\vv<i_\nu\xi,i_\nu\eta>dV_{\p M}
\end{equation}
for any $\xi,\eta\in \wedge^pV$.
\end{lem}
\begin{proof}
Let $\xi_1^{(p)},\cdots,\xi_{k-1}^{(p)}$ be the eigenforms for $v_1^{(p)},\cdots,v_{k-1}^{(p)}$ respectively that are orthogonal to each other. Let $\omega_1,\omega_2,\cdots,\omega_m$ be an orthonormal basis for $V$. That is to say,
\begin{equation}
\vv<\omega_i,\omega_j>=\delta_{ij}.
\end{equation}
Let $e_1,e_2,\cdots,e_m$ be the dual of $\omega_1,\omega_2,\cdots,\omega_m$ respectively. Let $H\subset \wedge^pV$ be the space of forms that are orthogonal to $\xi_1^{(p)},\xi_2^{(p)},\cdots,\xi_{k-1}^{(p)}$, and
\begin{equation}
S=\left\{\sum_{1\leq i_1<i_2<\cdots<i_p\leq m}a_{i_1i_2\cdots i_p}\omega_{i_1}
\wedge\omega_{i_2}\cdots\wedge\omega_{i_p}\bigg|\sum_{1\leq i_1<i_2<\cdots<i_p\leq m}a_{i_1i_2\cdots i_p}^2=1\right\}.
\end{equation}
Then, for any $\xi\in S\cap H$, we have
\begin{equation}
v_k \Vol(M)=v_k\int_M\|\xi\|^2dV_M\leq \int_{\p M}\|i_\nu \xi\|^2dV_{\p M}.
\end{equation}
Integrating the last inequality against $\xi$ over $S\cap H$, we have
\begin{equation}\label{eqn-int-1}
\begin{split}
v_k\Vol(M)\Vol(\sph^{C_m^p-k})\leq \int_{\p M}\int_{S\cap H}\|i_\nu \xi\|^2d\xi dV_{\p M}=\int_{\p M}\int_{S\cap H}\|i_{\nu^{\top}} \xi\|^2d\xi dV_{\p M}\\
\end{split}
\end{equation}
where $\nu^\top$ is the orthogonal projection of $\nu$ onto $\mbox{span}\{e_1,e_2,\cdots,e_m\}$.

Without loss of generality, we can suppose that  $\nu^\top=\nu_1e_1$ with $|\nu_1|\leq 1$. Then
\begin{equation}\label{eqn-int-2}
\begin{split}
\int_{S\cap H}\|i_{\nu^{\top}} \xi\|^2d\xi=&\nu_1^2\int_{S\cap H}\|i_{e_1} \xi\|^2d\xi\\
=&\nu_1^2\sum_{2\leq i_2<i_3<\cdots<i_p\leq m}\int_{S\cap H}a_{1i_2\cdots i_p}^2da\\
\leq&\frac{C_{m-1}^{p-1}\Vol(\sph^{C_m^p-k})}{C_m^p-k+1}
\end{split}
\end{equation}
by Lemma \ref{lem-int}. Substituting \eqref{eqn-int-2} into \eqref{eqn-int-1}, we obtain the conclusion.
\end{proof}
Now, combining Lemma \ref{lem-comp} and Lemma \ref{lem-v}, we obtain Theorem \ref{thm-para}. The inequality \eqref{eqn-Omega} in Theorem \ref{thm-Omega} is a direct corollary of Theorem \ref{thm-para}. The proof of the equality case in Theorem \ref{thm-Omega} is just the same as the proof the equality case in Theorem 5 of \cite{RS1}.

\section{Eigenvalue estimates on convex domains in $\R^n$}
In this section, we come to prove Theorem \ref{thm-convex}. The argument is just a simple use of coarea formula to estimate 
the integral in \eqref{eqn-int-1}.

\begin{proof}[Proof of Theorem \ref{thm-convex}]
Since $\Omega$ is strictly convex, $b_0(\Omega)=1$ and $b_p(\Omega)=0$ for $p=1,2,\cdots,n$. By Lemma \ref{lem-comp}, we only need to estimate the $v_k^{(p)}$ in Lemma \ref{lem-v} for this case.

Let $\xi_1^{(p)},\cdots,\xi_{k-1}^{(p)}$ be the eigenforms for $v_1^{(p)},\cdots,v_{k-1}^{(p)}$ respectively that are orthogonal to each other. Let
\begin{equation}
S=\left\{\sum_{1\leq i_1<i_2<\cdots<i_p\leq n}a_{i_1i_2\cdots i_p}dx_{i_1}
\wedge dx_{i_2}\cdots\wedge dx_{i_p}\bigg|\sum_{1\leq i_1<i_2<\cdots<i_p\leq n}a_{i_1i_2\cdots i_p}^2=1\right\},
\end{equation}
and $H\subset \wedge^p\R^n$ be the space of $p$-forms that are orthogonal to $\xi_1^{(p)},\xi_2^{(p)},\cdots,\xi_{k-1}^{(p)}$. Then, for any $\xi\in S\cap H$, we have
\begin{equation}
v_k^{(p)}\Vol(\Omega)=v_k^{(p)}\int_M\|\xi\|^2dV_M\leq \int_{\p M}\|i_\nu \xi\|^2dV_{\p M}.
\end{equation}
Integrating the last inequality against $\xi$ over $S\cap H$, we have
\begin{equation}\label{eqn-gauss}
\begin{split}
&v_k^{(p)}\Vol(\Omega)\Vol(\sph^{C_m^p-k})\\
\leq& \int_{\p \Omega}\int_{S\cap H}\|i_\nu \xi\|^2d\xi dV_{\p \Omega}\\
=&\int_{\p \Omega}\int_{S\cap H}\|i_{\nu} \xi\|^2d\xi dV_{\p \Omega}\\
=&\int_{S\cap H}\int_{\sph^{n-1}}\frac{\|i_{X} \xi\|^2}{K}dV_{\sph^{n-1}}(X)d\xi\\
\leq&\frac{1}{\min_{\p\Omega} K}\int_{S\cap H}\int_{\sph^{n-1}}\|i_{X} \xi\|^2dV_{\sph^{n-1}}(X)d\xi\\
=&\frac{1}{\min_{\p\Omega} K}\int_{S\cap H}\int_{\sph^{n-1}}\sum_{1\leq i_2<i_3<\cdots<i_p\leq n}\left(\sum_{i_1=1}^nX_{i_1}a_{i_1i_2\cdots i_{p}}\right)^2dV_{\sph^{n-1}}(X)da\\
=&\frac{1}{\min_{\p\Omega} K}\int_{S\cap H}\int_{\sph^{n-1}}\sum_{1\leq i_2<i_3<\cdots<i_p\leq n}\sum_{i_1=1}^nX_{i_1}^2a_{i_1i_2\cdots i_{p}}^2dV_{\sph^{n-1}}(X)da\\
=&\frac{V(\sph^{n-1})}{n\min_{\p \Omega}K}\int_{S\cap H}\sum_{1\leq i_2<i_3<\cdots<i_p\leq n}\sum_{i_1=1}^na_{i_1i_2\cdots i_{p}}^2da\\
=&\frac{pV(\sph^{n-1})}{n\min_{\p \Omega}K}\int_{S\cap H}\sum_{1\leq i_1< i_2<i_3<\cdots<i_p\leq n}a_{i_1i_2\cdots i_{p}}^2da\\
=&\frac{p\Vol(\sph^{n-1})\Vol(\sph^{C_m^p-k})}{n\min_{\p \Omega}K}.
\end{split}
\end{equation}
Here we have used the co-area formula and that the Jacobian of the Gauss map is the Gaussian curvature. Note that we have also skew symmetrically extended $a_{i_1i_2\cdots i_p}$ in their indices.

When equality of \eqref{eqn-gauss} is true, we know that $K\equiv \min _{\p\Omega} K$ is a constant. This implies that $\Omega$ is a ball. Combining this and Raulot-Savo's result \eqref{eqn-eigen-ball-1} and \eqref{eqn-eigen-ball-2} give us the conclusion.

\end{proof}

\end{document}